\documentclass[12pt]{article}
\usepackage{amsmath}

\usepackage{t1enc}
\usepackage[latin1]{inputenc}
\usepackage[english]{babel}
\usepackage{amsmath,amsthm}
\usepackage{amsfonts}
\usepackage{latexsym}
\usepackage[dvips]{graphicx}
\usepackage{graphicx}
\usepackage[natural]{xcolor}
\newtheorem{theorem}{Theorem}
\newtheorem{lemma}[theorem]{Lemma}
\newtheorem{proposition}[theorem]{Proposition}

\newtheorem{definition}[theorem]{Definition}

\newtheorem{remark}[theorem]{Remark}

\begin{document}

\date{30/11/2011}
\author{Liangquan Zhang$^{1,2}$ \thanks{
E-mail: xiaoquan51011@163.com.} \\
1. School of Mathematics, Shandong University\\
Jinan 250100, People's Republic of China.\\
{2. }Laboratoire de Math\'ematiques, \\
Universit\'e de Bretagne Occidentale, 29285 Brest C\'edex, France.}
\title{Large Deviation for Reflected Backward Stochastic Differential Equations}
\maketitle

\begin{abstract}
In this note, we prove the Freidlin-Wentzell's large deviation
principle for BSDEs with one-sided reflection.
\end{abstract}



\section{Introduction}

Backward stochastic differential equations (BSDEs in short) with
reflection were firstly studied by El Karoui, Kapoudjian, Pardoux,
Peng and Quenez in [3], which is, a standard BSDEs with an
additional continuous, increasing process to keep the solution above
a certain given continuous boundary process. This increasing process
must be chosen in certain minimal way, i.e. an integral condition,
called Skorohod reflecting condition, is satisfied. Besides, they
gave a probabilistic interpretation of viscosity solution of
variation inequality by the solution of reflected BSDEs. As we have
known, it can be used widely in mathematical finance, for example,
American options in incomplete market (see [6]).

On the other hand, the large deviation principle (LDP) characterizes
the limiting behavior of probability measure in term of rate
function which is a very active field in applied probability and
largely used in rare events simulation. Recently, there has been a
growing literature on studying the applications of LDP in finance
(see [7]).

We now consider the following small perturbation of reflected
forward and backward stochastic differential equations (1.1) and
(1.2)

\begin{equation}
X^{\varepsilon ,t,x}\left( s\right) =x+\int_t^sb\left(
r,X^{\varepsilon
,t,x}\left( r\right) \right) \text{d}r+\varepsilon ^{\frac 12}\int_t^s\text{d%
}W\left( r\right)  \tag{1.1}
\end{equation}
\begin{equation}
\left\{
\begin{array}{l}
Y^{\varepsilon ,t,x}\left( s\right) =g\left( X^{\varepsilon
,t,x}\left( T\right) \right) +\int_s^Tf\left( r,X^{\varepsilon
,t,x}\left( r\right) ,Y^{\varepsilon ,t,x}\left( r\right)
,Z^{\varepsilon ,t,x}\left( r\right)
\right) \text{d}r \\
\qquad \qquad +K^{\varepsilon ,t,x}\left( T\right) -K^{\varepsilon
,t,x}\left( r\right) -\int_s^TZ^{\varepsilon ,t,x}\left( r\right) \text{d}%
W\left( r\right) \\
Y^{\varepsilon ,t,x}\left( s\right) \geq h\left( s,X^{\varepsilon
,t,x}\left( s\right) \right) ,\qquad t\leq s\leq T \\
\int_t^T\left( Y^{\varepsilon ,t,x}\left( s\right) -h\left(
s,X^{\varepsilon ,t,x}\left( s\right) \right) \right)
dK^{\varepsilon ,t,x}\left( s\right) ,
\end{array}
\right.  \tag{1.2}
\end{equation}
The solution of this equation is denoted by
\[
\left( X^{\varepsilon ,t,x}\left( s\right) ,Y^{\varepsilon
,t,x}\left( s\right) ,Z^{\varepsilon ,t,x}\left( s\right)
,K^{\varepsilon ,t,x}\left( s\right) ,t\leq s\leq T\right) .
\]
We want to establish the large deviation principle of the law of $%
Y^{\varepsilon ,t,x}$ in the space of $C\left( \left[ 0,T\right] ;{\bf R}%
^n\right) ,$ namely the asymptotic estimates of probabilities
$P\left( Y^{\varepsilon ,t,x}\in \Gamma \right) ,$ where $\Gamma \in
{\cal B}\left( C\left( \left[ 0,T\right] ;{\bf R}^n\right) \right)
.$

In [8], Rainero first considered the same small random perturbation
for
BSDEs and obtained the Freidlin-Wentzell's large deviation estimates in $%
C\left( \left[ 0,T\right] ;{\bf R}^n\right) $ using the contraction
principle. Subsequently, Essaky in [4] investigated the large
deviation for BSDEs with subdifferential operator. It is necessary
to point out that BSDEs with subdifferential operator include as a
special case BSDEs whose solution is reflected at the boundary of a
convex subset of ${\bf R}^n$ (for more information see [5]).
Moreover, their convex function is fixed. Besides, in [4] $b$ does
not depends on time variable. From this viewpoint, our work cannot
be covered by their results.

In Section 2, we give the framework of our paper. Here we review the
basic concepts of large deviation and assumptions on (1.1) and
(1.2). Then in
Section 3 we show our main result Theorem 8. Throughout the paper, $C$ and $%
K,$ with or without indexes will denote different constants changing
from line to line whose values are not important.

\section{Preliminaries}

Let us begin by introducing the setting for the stochastic
differential
differential we want to investigate. Consider as Brownian motion $W$ is the $%
n$-dimensional coordinate process on the classical Wiener space
$\left( \Omega ,{\cal F},P\right) $, i.e., $\Omega $ is the set of
continuous
functions from $\left[ 0,T\right] $ to ${\bf R}^n$ starting from $0$ ($%
\Omega =C\left( \left[ 0,T\right] ;{\bf R}^n\right) $, ${\cal F}$
the completed Borel $\sigma $-algebra over $\Omega $, $P$ the Wiener
measure and $W$ the canonical process: $W_s\left( \omega \right)
=\omega _s$, $s\in \left[ 0,T\right] ,$ $\omega \in \Omega .$ By
$\left\{ {\cal F}_s,0\leq s<T\right\} $ we denote the natural
filtration generated by $\left\{ W_s\right\} _{0\leq s<T}$ and
augmented by all $P$-null sets, i.e.,
\[
{\cal F}_s=\sigma \left\{ W_r,r\leq s\right\} \vee {\cal N}_p,\text{
}s\in \left[ 0,T\right] ,
\]
where ${\cal N}_p$ is the set of all $P$-null subsets. For any $n\geq 1,$ $%
\left| z\right| $ denotes the Euclidean norm of $z\in {\bf R}^n$.

$b,g$, $f$ and $h$ in Eq. (1.1) and (1.2) are defined as follows.
First, let $b:\left[ 0,T\right] \times {\bf R}^n\rightarrow {\bf
R}^n$ be continuous mapping and satisfy linear growth, which are
Lipschitz with respect to their second variable, uniformly with
respect to $t\in \left[ 0,T\right] .$ Second, $g\in C\left( {\bf
R}^n\right) $ and has at most polynomial growth
at infinity and satisfies Lipschitz condition. For $f$, assume that $%
f:\left[ 0,T\right] \times {\bf R}^n\times {\bf R\times R}^n\rightarrow {\bf %
R}$ is jointly continuous and for some $K>0,$ admits that
\begin{equation}
\left| f\left( t,x,0,0\right) \right| \leq K\left( 1+\left| x\right|
^2\right) ,  \tag{2.1}
\end{equation}

\begin{equation}
\left| f\left( t,x,y,z\right) -f\left( t,x^{^{\prime }},y^{^{\prime
}},z^{^{\prime }}\right) \right| \leq K\left( \left| x-x^{^{\prime
}}\right| +\left| y-y^{^{\prime }}\right| +\left| z-z^{^{\prime
}}\right| \right) , \tag{2.2}
\end{equation}
for $t\in \left[ 0,T\right] ,$ $x,$ $x^{^{\prime }},$ $z,$
$z^{^{\prime }}\in {\bf R}^n,$ $y,$ $y^{^{\prime }}\in {\bf R.}$
Finally, $h:\left[ 0,T\right] \times {\bf R}^n\rightarrow {\bf R}$
is jointly continuous in $t$ and $x$ and satisfies
\begin{equation}
h\left( t,x\right) \leq K\left( 1+\left| x\right| \right) ,\qquad
t\in \left[ 0,T\right] ,\text{ }x\in {\bf R}^n,  \tag{2.3}
\end{equation}

\begin{equation}
\left| h\left( t,x\right) -h\left( t,x^{^{\prime }}\right) \right|
\leq K\left| x-x^{^{\prime }}\right| ,\text{ }x,x^{^{\prime }}\in
{\bf R}^n. \tag{2.4}
\end{equation}
We assume moreover that $g\left( x\right) \geq h\left( T,x\right) ,$
$x\in {\bf R}^n$

For each $t>0,$ we denote by $\left\{ {\cal F}_s^t,\text{ }t\leq
s\leq
T\right\} $ the natural filtration of the Brownian motion $\left\{ B_s-B_t,%
\text{ }t\leq s\leq T\right\} ,$ argumented by the $P$-null set of ${\cal F}$%
. It follows from Theorem 5.2 in [3] that, under the above
assumptions, there exists a unique triple $\left( Y^{\varepsilon
,t,x},Z^{\varepsilon ,t,x},K^{\varepsilon ,t,x}\right) $ of $\left\{
{\cal F}_s^t\right\} $ progressively measurable processes, which
solves Eq. (1.1) and (1.2).

Now we give two more accurate estimates on the norm of the solution
similar to Proposition 3.5 in [3].

\begin{lemma}
Let $\left( X^{\varepsilon ,t,x}\left( s\right) ,Y^{\varepsilon
,t,x}\left( s\right) ,Z^{\varepsilon ,t,x}\left( s\right)
,K^{\varepsilon ,t,x}\left( s\right) \quad t\leq s\leq T,\text{
}\varepsilon >0\right) $ be the solution of the above reflected
FBSDE (1.1), (2.2), then there exists a constant C such that
\begin{eqnarray*}
\ \ {\Bbb E}\left[ \sup\limits_{t\leq s\leq T}\left| Y^{\varepsilon
,t,x}\left( s\right) \right| ^2+\int_0^T\left| Z^{\varepsilon
,t,x}\left( s\right) \right| ^2+\left( K^{\varepsilon ,t,x}\right)
^2\left( T\right)
\right]  \\
\ \leq C{\Bbb E}\left[ g^2\left( X^{\varepsilon ,t,x}\left( T\right)
\right)
+\int_t^Tf^2\left( s,X^{\varepsilon ,t,x}\left( s\right) ,0,0\right) \text{d}%
t+\sup\limits_{t\leq s\leq T}h^2\left( t,X^{\varepsilon ,t,x}\left(
s\right) \right) \right]
\end{eqnarray*}
\begin{equation}
\tag{2.5}
\end{equation}
\end{lemma}

We now consider the following related obstacle problem for a
parabolic PDEs. More precisely, a solution of the obstacle problem
is a function $u:\left[ 0,T\right] \times {\bf R}^n\rightarrow {\bf
R}$ which satisfies
\begin{equation}
\left\{
\begin{array}{l}
\min \left( u\left( t,x\right) -h\left( t,x\right) ,-\frac{\partial u}{%
\partial t}\left( t,x\right) -{\cal L}_tu\left( t,x\right) -f\left(
t,x,u\left( t,x\right) ,\left( \nabla u\varepsilon ^{\frac
12}\right) \left(
t,x\right) \right) \right) =0, \\
u\left( T,x\right) =g\left( x\right) ,\qquad \left( t,x\right) \in
\left( 0,T\right) \times {\bf R}^n,\text{ }x\in {\bf R}^n,
\end{array}
\right.   \tag{2.6}
\end{equation}
where
\[
{\cal L}_t=\frac \varepsilon 2\sum_{i,j=1}^n\frac{\partial
^2}{\partial x_i\partial x_j}+\sum_{i=1}^nb_i\left( t,x\right) \frac
\partial {\partial x_i}.
\]
Now define
\begin{equation}
u^\varepsilon \left( t,x\right) =Y^{\varepsilon ,t,x}\left( t\right)
,\qquad \left( t,x\right) \in \left[ 0,T\right] \times {\bf
R}^n,\text{ }\varepsilon
>0,  \tag{2.7}
\end{equation}
which is a deterministic quantity since it is ${\cal F}_t^t$
measurable.

\begin{definition}
Let $t\in \left[ 0,T\right] ,$ we define the mapping $G^\varepsilon
:C\left(
\left[ t,T\right] :{\bf R}^n\right) \rightarrow C\left( \left[ t,T\right] :%
{\bf R}^n\right) $ by
\[
G^\varepsilon \left( \psi \right) =\left[ t\rightarrow u^\varepsilon
\left(
t,\psi \left( t\right) \right) \right] ,\quad 0\leq t\leq s\leq T,\text{ }%
\psi \in C\left( \left[ t,T\right] :{\bf R}^n\right) ,
\]
where $u^\varepsilon $ is given by (2.7).
\end{definition}

\noindent Immediately, we have
\[
Y^{\varepsilon ,t,x}\left( t\right) =G^\varepsilon \left(
X^{\varepsilon ,t,x}\right) \left( t\right) .
\]
From now on, for $\varepsilon =0,$ $u$ and $G$ stand for $u^0$ and
$G^0.$

\noindent In order to prove the uniform convergence of the mapping $%
G^\varepsilon $, which will be shown in next section, we need to
estimate the following formula
\[
\left\| G^\varepsilon \left( \varphi \right) -G\left( \varphi
\right) \right\| =\sup\limits_{0\leq s\leq T}\left| u^\varepsilon
\left( s,\varphi \left( s\right) \right) -u\left( s,\varphi \left(
s\right) \right) \right| ,\quad \varphi \in C\left( \left[
0,T\right] :{\bf R}^n\right)
\]

\noindent or
\[
\left\| G^\varepsilon \left( \varphi \right) -G\left( \varphi
\right) \right\| =\sup\limits_{0\leq s\leq T}\left| Y^{\varepsilon
,s,\varphi \left( s\right) }\left( s\right) -Y^{s,\varphi \left(
s\right) }\left( s\right) \right|
\]

\begin{proposition}
Under the above assumptions (2.1)-(2.4), $u^\varepsilon $ is a
viscosity solution of the obstacle problem (2.6).
\end{proposition}

The proof can be seen in [3]. Before giving a large deviation
principle for SDEs (1.1), we recall the following definitions.

\begin{definition}
If $E$ is a complete separable metric space, then a function ${\cal
I}$ defined on $E$ is called a rate function if it has the following
properties:
\begin{equation}
\left\{
\begin{array}{l}
\text{(a)\qquad }{\cal I}:E\rightarrow \left[ 0,+\infty \right] ,\text{ }%
{\cal I}\text{ is lower semicontinuous.} \\
\text{(b)\qquad If }0\leq a\leq \infty ,\text{ then }C_I\left(
a\right) =\left\{ x\in E:I\left( x\right) \leq a\right\} \text{ is
compact.}
\end{array}
\right.   \tag{2.8}
\end{equation}
\end{definition}

\begin{definition}
If $E$ is a complete separable metric space, ${\cal B}$ is the Borel
$\sigma $-field on $E,$ $\left\{ \mu _\varepsilon :\varepsilon
>0\right\} $ is a family of probability measure on $\left( E,{\cal
B}\right) ,$ and ${\cal I}$ is a function defined on $E$ and
satisfying (2.8), then we say that $\left\{ \mu _\varepsilon
\right\} _{\varepsilon >0}$ satisfies a large deviation principle
with rate ${\cal I}$ if:
\begin{equation}
\left\{
\begin{array}{l}
\text{(a)\qquad For every open subset }A\text{ of }E\text{,} \\
\qquad \qquad \liminf_{\varepsilon \rightarrow 0}\varepsilon \log
\mu
_\varepsilon \left( A\right) \geq -{\cal I}\left( A\right) . \\
\text{(b)\qquad For every closed subset }A\text{ of }E, \\
\qquad \qquad \limsup_{\varepsilon \rightarrow 0}\varepsilon \log
\mu _\varepsilon \left( A\right) \leq -{\cal I}\left( A\right) ,
\end{array}
\right.   \tag{2.9}
\end{equation}
where, and below, if ${\cal I}$ is a function defined on the set $E$
and $A$ is a subset of $E$, then ${\cal I}\left( A\right) $ is
defined to be the infimum of ${\cal I}$ on $A$. Unless otherwise
stated, all lim infs and sups are as $\varepsilon \rightarrow 0.$
\end{definition}

In the work of Michelle Boue and Paul Dupuis [1], they proved the
large deviation principle for the solution of the SDEs (1.1) as
follows:

\begin{lemma}
The process $X^{\varepsilon ,t,x}$ given by (1.1) satisfies a large
deviation principle in $C\left( \left[ 0,T\right] :{\bf R}^n\right)
$ with rate function ${\cal I}$ defined by
\[
{\cal I}_{t,x}\left( \xi \right) =\inf\limits_{\left\{ v\in
L^2\left( \left[ 0,T\right] :{\bf R}^n\right) :\xi \left( t\right)
=x+\int_t^sb\left( r,\xi \left( r\right) \right)
\text{d}r+\int_t^sv\left( r\right) \text{d}r\right\} }\frac
12\int_t^T\left\| v\left( s\right) \right\| \text{d}s,
\]
whenever $\left\{ v\in L^2\left( \left[ 0,T\right] :{\bf R}^n\right)
:\xi
\left( t\right) =x+\int_t^sb\left( r,\xi \left( r\right) \right) \text{d}%
r+\int_t^sv\left( r\right) \text{d}r\right\} \neq \emptyset ,$ and ${\cal I}%
_{t,x}\left( \xi \right) =\infty $ otherwise.
\end{lemma}

\begin{remark}
Note that Laplace principle is equivalent to a large deviation
principle if the definition of a rate function includes the
requirement of compact level sets.
\end{remark}

\section{Main Result}

We have the following:

\begin{theorem}
Under the assumptions (2.1)-(2.4), $Y^{\varepsilon ,t,x}$ satisfies
a large deviation principle with a rate function
\begin{equation}
\widetilde{{\cal I}}_{t,x}\left( \tilde \xi \right) =\inf \left\{
\left. {\cal I}_{t,x}\left( \xi \right) \right| \tilde \xi \left(
t\right) =G\left(
\xi \right) \left( t\right) =u\left( t,\xi \left( t\right) \right) ,\text{ }%
t\in \left[ 0,T\right] ,\text{ }\xi \in C\left( \left[ 0,T\right] :{\bf R}%
^n\right) \right\} .  \tag{3.1}
\end{equation}
\end{theorem}

\noindent For the proof of this theorem we need four auxiliary lemmata. Let $%
\chi ^{t,x}$ be the solution of the following deterministic equation
\begin{equation}
\chi ^{t,x}\left( s\right) =x+\int_t^sb\left( r,\chi ^{t,x}\left(
r\right) \right) dr.  \tag{3.2}
\end{equation}
We have the following

\begin{lemma}
For all $\varepsilon \in \left( 0,1\right] $, there exists a
constant $C>0$, independent of $x$, $t$ and $\varepsilon $, such
that
\begin{equation}
{\Bbb E}\left[ \sup\limits_{t\leq s\leq T}\left| X^{\varepsilon
,t,x}\left( t\right) -\chi ^{t,x}\left( t\right) \right| \right]
\leq C\varepsilon . \tag{3.3}
\end{equation}
\end{lemma}

\begin{proof} Applying It\^o's formula $\left( X^{\varepsilon
,t,x}\left( s\right) -\chi ^{t,x}\left( s\right) \right) ^2$ on
$\left[ t,T\right] ,$ we have
\begin{eqnarray*}
\left( X^{\varepsilon ,t,x}\left( s\right) -\chi ^{t,x}\left(
s\right) \right) ^2 &=&2\int_t^s\left( X^{\varepsilon ,t,x}\left(
r\right) -\chi ^{t,x}\left( r\right) \right) \left[ b\left(
r,X^{\varepsilon ,t,x}\left(
r\right) \right) -b\left( r,\chi ^{t,x}\left( r\right) \right) \right] \text{%
d}r \\
&&+2\sqrt{\varepsilon }\int_t^s\left( X^{\varepsilon ,t,x}\left(
r\right) -\chi ^{t,x}\left( r\right) \right) \text{d}W\left(
r\right) +\varepsilon
\int_t^s\text{d}r \\
&\leq &2K\int_t^s\left| X^{\varepsilon ,t,x}\left( r\right) -\chi
^{t,x}\left( r\right) \right| ^2\text{d}r+\varepsilon T \\
&&+2\sqrt{\varepsilon }\int_t^s\left( X^{\varepsilon ,t,x}\left(
r\right) -\chi ^{t,x}\left( r\right) \right) \text{d}W\left(
r\right) ,
\end{eqnarray*}
\begin{equation}
\tag{3.4}
\end{equation}
since $b$ satisfies Lipchitz condition with Lipchitz constant $K.$
Taking expectation on both sides of (3.4), we have
\[
{\Bbb E}\left[ \left| X^{\varepsilon ,t,x}\left( s\right) -\chi
^{t,x}\left( s\right) \right| ^2\right] \leq 2K\int_t^s\left|
X^{\varepsilon ,t,x}\left( r\right) -\chi ^{t,x}\left( r\right)
\right| ^2\text{d}r+\varepsilon T.
\]
It follows from Gronwall inequality that
\[
{\Bbb E}\left[ \left| X^{\varepsilon ,t,x}\left( s\right) -\chi
^{t,x}\left( s\right) \right| ^2\right] \leq \varepsilon KT\exp
\left( 2K\left( s-t\right) \right)
\]
Hence
\begin{eqnarray*}
{\Bbb E}\left[ \sup\limits_{t\leq s\leq T}\left| X^{\varepsilon
,t,x}\left(
s\right) -\chi ^{t,x}\left( s\right) \right| ^2\right] &\leq &2K\int_t^T%
{\Bbb E}\left[ \left| X^{\varepsilon ,t,x}\left( r\right) -\chi
^{t,x}\left(
r\right) \right| ^2\right] \text{d}r+\varepsilon T \\
&&{\Bbb E}\left[ \sup\limits_{t\leq s\leq T}2\sqrt{\varepsilon }%
\int_t^s\left( X^{\varepsilon ,t,x}\left( r\right) -\chi
^{t,x}\left(
r\right) \right) \text{d}W\left( r\right) \right] \\
&\leq &\varepsilon \left( KT\exp \left( 2K\left( T-t\right) \right)
+T\right)
\\
&&+{\Bbb E}\left[ \sup\limits_{t\leq s\leq T}2\sqrt{\varepsilon }%
\int_t^s\left( X^{\varepsilon ,t,x}\left( r\right) -\chi
^{t,x}\left( r\right) \right) \text{d}W\left( r\right) \right] .
\end{eqnarray*}
It follows from the B-D-G inequality that there exists a constant
$K_1$ such that
\begin{eqnarray*}
{\Bbb E}\left[ \sup\limits_{t\leq s\leq T}\left| X^{\varepsilon
,t,x}\left( s\right) -\chi ^{t,x}\left( s\right) \right| ^2\right]
&\leq &\varepsilon
\left( KT\exp \left( 2K\left( T-t\right) \right) +T\right) \\
&&+2K_1\sqrt{\varepsilon }\left( \alpha +\frac 1\alpha {\Bbb
E}\left[ \sup\limits_{t\leq s\leq T}\left| X^{\varepsilon
,t,x}\left( s\right) -\chi ^{t,x}\left( s\right) \right| ^2\right]
\right) ,
\end{eqnarray*}
where $\alpha $ is a positive constant to be determined.

Consequently, choosing some $\alpha $ we obtain that
\[
{\Bbb E}\left[ \sup\limits_{t\leq s\leq T}\left| X^{\varepsilon
,t,x}\left( s\right) -\chi ^{t,x}\left( s\right) \right| ^2\right]
\leq \varepsilon K_2,
\]
where $K_2$ only depends on $T,K,K_1$ and is independent of
$\varepsilon ,t,x.$ \end{proof}

\begin{remark}
As a consequence of Lemma 7, the solution of the $X^{\varepsilon
,t,x}$ converges to the deterministic path $\chi ^{t,x}$ in $L^2$.
\end{remark}

Now consider the following deterministic equations
\[
\left\{
\begin{array}{l}
\chi ^{t,x}\left( s\right) =x+\int_t^sb\left( r,\chi ^{t,x}\left(
r\right)
\right) dr. \\
Y^{t,x}\left( s\right) =g\left( X^{t,x}\left( T\right) \right)
+\int_s^Tf\left( r,X^{t,x}\left( r\right) ,Y^{t,x}\left( r\right)
,0\right)
\text{d}r \\
\qquad \qquad +K^{t,x}\left( T\right) -K^{t,x}\left( r\right) \\
Y^{t,x}\left( s\right) \geq h\left( s,X^{t,x}\left( s\right) \right)
,\qquad
t\leq s\leq T \\
\int_t^T\left( Y^{t,x}\left( s\right) -h\left( s,X^{t,x}\left(
s\right) \right) \right) dK^{t,x}\left( s\right) ,
\end{array}
\right.
\]
\begin{equation}
\tag{3.5}
\end{equation}
We have the following Lemma

\begin{lemma}
For all $\varepsilon \in \left( 0,1\right] $, there exists a
constant $C>0$, independent of $x$, $t$ and $\varepsilon $, such
that
\begin{eqnarray*}
\ {\Bbb E}\left[ \sup\limits_{t\leq s\leq T}\left| Y^{t,x}\left(
s\right) -Y^{\varepsilon ,t,x}\left( s\right) \right| ^2\right]
+{\Bbb E}\left[
\int_t^T\left| Z^{\varepsilon ,t,x}\left( s\right) \right| ^2\text{d}%
s\right]  \\
\ \leq C{\Bbb E}\left[ \left( \sup\limits_{t\leq s\leq T}\left|
\widehat{\Xi }^{\varepsilon ,t,x}\left( s\right) \right| ^2\right)
^{\frac 12}+\left| \widehat{\Xi }^{\varepsilon ,t,x}\left( T\right)
\right| ^2+\int_t^T\left| \widehat{\Xi }^{\varepsilon ,t,x}\left(
s\right) \right| ^2\text{d}s\right]
\end{eqnarray*}
\begin{equation}
\tag{3.6}
\end{equation}
where $\widehat{\Xi }^{\varepsilon ,t,x}\left( s\right)
=X^{\varepsilon ,t,x}\left( s\right) -\chi ^{t,x}\left( s\right) ,$
$0\leq t\leq s\leq T,$ $P $-a.s.
\end{lemma}

\begin{proof} It follows from It\^o's formula that
\begin{eqnarray*}
&&{\Bbb E}\left[ \left| \left| Y^{\varepsilon ,t,x}\left( s\right)
-Y^{t,x}\left( s\right) \right| ^2\right| \right] +{\Bbb E}\left[
\int_t^T\left| Z^{\varepsilon ,t,x}\left( s\right) \right|
^2\text{d}s\right]
\\
&=&{\Bbb E}\left[ g\left( X^{\varepsilon ,t,x}\left( T\right)
\right)
-g\left( \chi ^{t,x}\left( T\right) \right) \right] +2{\Bbb E}\left[ \int_t^T%
\widehat{f}\left( s\right) \left( Y^{\varepsilon ,t,x}\left(
s\right)
-Y^{t,x}\left( s\right) \right) \text{d}s\right] \\
&&+2{\Bbb E}\left[ \int_t^T\left( Y^{\varepsilon ,t,x}\left(
s\right) -Y^{t,x}\left( s\right) \right) \text{d}\left(
K^{\varepsilon ,t,x}\left( s\right) -K^{t,x}\left( s\right) \right)
\right] ,
\end{eqnarray*}
where
\[
\widehat{f}\left( s\right) =f\left( s,X^{\varepsilon ,t,x}\left(
s\right) ,Y^{\varepsilon ,t,x}\left( s\right) ,Z^{\varepsilon
,t,x}\left( s\right) \right) -f\left( s,\chi ^{t,x}\left( s\right)
,Y^{t,x}\left( s\right) ,0\right) .
\]
Hence
\begin{eqnarray*}
&&{\Bbb E}\left[ \left| \left| Y^{\varepsilon ,t,x}\left( s\right)
-Y^{t,x}\left( s\right) \right| ^2\right| \right] +{\Bbb E}\left[
\int_t^T\left| Z^{\varepsilon ,t,x}\left( s\right) \right|
^2\text{d}s\right]
\\
&\leq &K{\Bbb E}\left[ \left| X^{\varepsilon ,t,x}\left( T\right)
-\chi
^{t,x}\left( T\right) \right| ^2\right] \\
&&+2K{\Bbb E}\left[ \int_t^T\left[ \left| X^{\varepsilon ,t,x}\left(
s\right) -\chi ^{t,x}\left( s\right) \right| +\left| Y^{\varepsilon
,t,x}\left( s\right) -Y^{t,x}\left( s\right) \right| +\left|
Z^{\varepsilon ,t,x}\left( s\right) \right| \right] \left(
Y^{\varepsilon ,t,x}\left(
s\right) -Y^{t,x}\left( s\right) \right) \right] \text{d}s \\
&&+2{\Bbb E}\left[ \int_t^T\left( Y^{\varepsilon ,t,x}\left(
s\right) -Y^{t,x}\left( s\right) \right) \text{d}\left(
K^{\varepsilon ,t,x}\left( s\right) -K^{t,x}\left( s\right) \right)
\right] .
\end{eqnarray*}
\begin{equation}
\tag{3.7}
\end{equation}
Since that $K^{\varepsilon ,t,x}$ (respectively $K^{t,x}$) grows when $%
Y^{\varepsilon ,t,x}\left( s\right) =h\left( s,X^{\varepsilon
,t,x}\left( s\right) \right) $ (respectively $Y^{t,x}\left( s\right)
=h\left( s,X^{t,x}\left( s\right) \right) $) only, and
$Y^{\varepsilon ,t,x}\left(
s\right) \geq h\left( s,X^{\varepsilon ,t,x}\left( s\right) \right) $ and $%
Y^{t,x}\left( s\right) \geq h\left( s,\chi ^{t,x}\left( s\right)
\right) $ for $s\in \left[ t,T\right] ,$ $P$-a.s.
\begin{eqnarray*}
&&\int_t^T\left( Y^{\varepsilon ,t,x}\left( s\right) -Y^{t,x}\left(
s\right) \right) \text{d}\left( K^{\varepsilon ,t,x}\left( s\right)
-K^{t,x}\left(
s\right) \right) \\
&=&\int_t^T\left( h\left( s,X^{\varepsilon ,t,x}\left( s\right)
\right) -Y^{t,x}\left( s\right) \right) \text{d}K^{\varepsilon
,t,x}+\int_t^T\left( h\left( s,\chi ^{t,x}\left( s\right) \right)
-Y^{\varepsilon ,t,x}\left(
s\right) \right) \text{d}K^{t,x}\left( s\right) \\
&\leq &\int_t^T\left| h\left( s,X^{\varepsilon ,t,x}\left( s\right)
\right) -h\left( s,\chi ^{t,x}\left( s\right) \right) \right|
\text{d}\left(
K^{\varepsilon ,t,x}\left( s\right) +K^{t,x}\left( s\right) \right) \\
&\leq &K\int_t^T\left| X^{\varepsilon ,t,x}\left( s\right) -\chi
^{t,x}\left( s\right) \right| \text{d}\left( K^{\varepsilon
,t,x}\left(
s\right) +K^{t,x}\left( s\right) \right) \\
&\leq &K\sup\limits_{t\leq s\leq T}\left| X^{\varepsilon ,t,x}\left(
s\right) -\chi ^{t,x}\left( s\right) \right| \int_t^T\text{d}\left(
K^{\varepsilon ,t,x}\left( s\right) +K^{t,x}\left( s\right) \right)
\end{eqnarray*}
\begin{equation}
\tag{3.8}
\end{equation}
Taking the expectation on two sides of (3.8), we have
\begin{eqnarray*}
&&{\Bbb E}\left[ \int_t^T\left( Y^{\varepsilon ,t,x}\left( s\right)
-Y^{t,x}\left( s\right) \right) \text{d}\left( K^{\varepsilon
,t,x}\left(
s\right) -K^{t,x}\left( s\right) \right) \right]  \\
&\leq &K\left( {\Bbb E}\left[ \left| \sup\limits_{t\leq s\leq
T}\left| X^{\varepsilon ,t,x}\left( s\right) -\chi ^{t,x}\left(
s\right) \right|
\right| ^2\right] \right) ^{\frac 12}\left( {\Bbb E}\left[ \left| \int_t^T%
\text{d}\left( K^{\varepsilon ,t,x}\left( s\right) +K^{t,x}\left(
s\right)
\right) \right| ^2\right] \right) ^{\frac 12} \\
&\leq &K\left( {\Bbb E}\left[ \sup\limits_{t\leq s\leq T}\left|
X^{\varepsilon ,t,x}\left( s\right) -\chi ^{t,x}\left( s\right)
\right| ^2\right] \right) ^{\frac 12}\left( {\Bbb E}\left[ \left(
K^{\varepsilon ,t,x}\right) ^2\left( T\right) +\left( K^{t,x}\right)
^2\left( T\right) \right] \right) ^{\frac 12}
\end{eqnarray*}
It follows from the assumption on $b,$ $\varepsilon \in \left(
0,1\right\rceil $ and Lemma 1 that ${\Bbb E}\left[ \left(
K^{\varepsilon ,t,x}\left( T\right) \right) ^2+\left( K^{t,x}\left(
T\right) \right)
^2\right] $ is bounded denoted by a positive constant $M$ independent of $%
\varepsilon .$ Now we turn back to (3.7). Thanks to the inequality
$ab\leq \frac{a^2+b^2}2$ and Grownwall's lemma$,$ we deduce
immediately that
\begin{eqnarray*}
&&{\Bbb E}\left[ \left| \left| Y^{\varepsilon ,t,x}\left( s\right)
-Y^{t,x}\left( s\right) \right| ^2\right| \right] +\frac 12{\Bbb
E}\left[ \int_t^T\left| Z^{\varepsilon ,t,x}\left( s\right) \right|
^2\text{d}s\right]
\\
&\leq &K^{^{\prime }}{\Bbb E}\left[ \left( \sup\limits_{t\leq s\leq
T}\left| \widehat{\Xi }^{\varepsilon ,t,x}\left( s\right) \right|
^2\right) ^{\frac 12}+\left| \widehat{\Xi }^{\varepsilon ,t,x}\left(
T\right) \right|
^2+\int_t^T\left| \widehat{\Xi }^{\varepsilon ,t,x}\left( s\right) \right| ^2%
\text{d}s\right] ,
\end{eqnarray*}
where $K^{^{\prime }}$ depends on $K$, $M$ and $T.$ The proof is
complete.
\end{proof}

\begin{remark}
As a consequence of Lemmas 9 and 11, we get
\[
\left\| G^\varepsilon \left( \varphi \right) -G\left( \varphi
\right) \right\| \leq \sqrt{\varepsilon }K^{^{^{\prime \prime
}}},\text{ }\varphi \in C\left( \left[ 0,T\right] :{\bf R}^n\right)
\]
where $K^{^{^{\prime \prime }}}$ is a constant related to
$K^{^{\prime }}.$
\end{remark}

Now we are able to give the proof of Theorem 8:

\begin{proof}By virtue of the contraction principle Theorem
4.2.23, page 133 in [2], we just need to show that $G^\varepsilon $, $%
\varepsilon \in \left( 0,1\right\rceil $ are continuous and $\left\{
G^\varepsilon \right\} $ converges uniformly to $G$ on every compact of $%
C\left( \left[ 0,T\right] :{\bf R}^n\right) ,$ as $\varepsilon $
tends to zero. As a matter of fact, since $u^\varepsilon $ is
continuous by Proposition 3, it is not hard to prove that
$G^\varepsilon $ is also
continuous. Next let us show the uniform convergence of the mapping $%
G^\varepsilon .$ It follows from Remark 12. \end{proof}


\begin{thebibliography}{1}
\bibitem[1]{BD}  Boue, M., Dupuis, P., 1998. A variation representation for
certain functionals of Brownian motion. Annals of Probability. Vol.
26. No. 4, 1641-1659.

\bibitem[2]{DZ}  Dembo, A., Zeitouni, O., 1998. Large Deviations Techniques
and Applications, second ed., Springer-Verlag, New York.

\bibitem[3]{f}  El-Karoui, N., Kapoudjian, C., Pardoux, E., Peng, S.,
Quenez, M.C., 1997. Reflected solutions of backward SDE's and
related obstacle problems for PDE's. Annals of Probability. 25, 2,
702-737.

\bibitem[4]{E}  Essaky, El H., 2008. Large deviation principle for a
backward stochastic differential equation with subdifferential
operator. C. R. Acad. Sci. Paris, Ser. I 346 75-78.

\bibitem[5]{R}  Pardoux, E., Rascanu, A., 1998. Backward stochastic
differential equations subdifferential operator and related
variation inequality. Stochastic Process Appl. 76 (2) 191-215.

\bibitem[6]{P}  Peng, S., Xu, M., 2010. Reflected BSDE with a constraint and
its applications in an incomplete market. Bernoulli Volume 16,
Number 3 614-640.

\bibitem[7]{P}  Pham, H., 2007. Some applications and methods of large
deviations in finance and insurance. arXiv:math/0702473v2 [math.PR].

\bibitem[8]{r}  Rainero, S., 2006. Un principe de grandes d\'eviations pour
une \'equation diff\'erentielle stochastique progressive
r\'etrograde. C. R. Acad. Sci Paris, Ser. I 343 (2) 141-144.
\end{thebibliography}
\end{document}